\documentclass[10pt,fleqn]{amsart}

\usepackage{amsmath,amssymb,latexsym}
\usepackage[mathscr]{eucal}
\usepackage{accents}

\usepackage{cite}

\theoremstyle{plain}
\newtheorem{theorem}{Theorem}
\newtheorem{lemma}[theorem]{Lemma}

\numberwithin{equation}{section}

\def\mydate{\number\year-\ifnum\month<10{0}\fi\number\month-\ifnum\day<10{0}\fi\number\day}

\newcommand{\dy}{\partial}
\newcommand{\ddt}[1]{\frac{\mathrm{d}{#1}}{\mathrm{d}{t}}}

\newcommand{\sfrac}[2]{{\textstyle\frac{#1}{#2}}}
\newcommand{\tssum}{{\textstyle\sum}}

\newcommand{\Zahl}{\mathbb{Z}}
\newcommand{\Real}{\mathbb{R}}

\newcommand{\ex}{\mathrm{e}}
\newcommand{\im}{\mathrm{i}}
\newcommand{\eps}{\varepsilon}
\newcommand{\vfi}{\varphi}
\newcommand{\vtt}{\vartheta}

\newcommand{\gb}{\nabla}
\newcommand{\sgb}{\nabla^\perp}

\newcommand{\sump}[1]{\mathop{\smash{\mathop{{\sum}_{#1}'}}{\vphantom\sum}}}

\newcommand{\Dom}{D}

\newcommand{\kpg}{\kappa_g}
\newcommand{\Proj}{{\sf P}}
\newcommand{\Zu}{\mathbb{Z}_+^2}

\newcommand{\tht}{\theta}
\newcommand{\ds}{\;\mathrm{d}s}
\newcommand{\dtau}{\;\mathrm{d}\tau}
\newcommand{\dsig}{\;\mathrm{d}\sigma}

\newcommand{\Unif}{\mathcal{U}}

\newcommand{\EV}{{\sf E}}
\newcommand{\Var}{{\sf var}}
\newcommand{\cov}{{\sf cov}}
\newcommand{\cis}{\mathrm{cis\,}}

\newcommand{\Pkk}{\Proj_{\kappa,2\kappa}}

\newcommand{\dr}{\;\mathrm{d}r}

\newcommand{\iPhi}{\hat\Phi}

\begin{document}

\title[Batchelor--Howells--Townsend spectrum]
{Tracer turbulence: the Batchelor--Howells--Townsend spectrum revisited}
\author{M. S. Jolly$^{1}$}
\address{$^1$Department of Mathematics\\
Indiana University\\ Bloomington, IN 47405}
\author{D. Wirosoetisno$^{2}$}
\address{$^2$Department of Mathematics Sciences\\
Durham University\\ Durham, United Kingdom\ \ DH1 3LE}
\email[M. S. Jolly]{msjolly@indiana.edu}
\email[D. Wirosoetisno]{djoko.wirosoetisno@durham.ac.uk}

\thanks{This work was supported in part by NSF grant number DMS-1818754, and the Leverhulme Trust grant VP1-2015-036.  The authors thank J.~Vanneste and O.~Hryniv for helpful discussions.}



\begin{abstract}
Given a velocity field $u(x,t)$, we consider the evolution of a passive tracer $\tht$ governed by $\dy_t\tht + u\cdot\gb\tht = \Delta\tht + g$ with time-independent source $g(x)$.
When $\|u\|$ is small, Batchelor, Howells and Townsend (1959, J.\ Fluid Mech.\ 5:134) predicted that the tracer spectrum scales as $|\tht_k|^2\propto|k|^{-4}|u_k|^2$.
In this paper, we prove that this scaling does indeed hold for large $|k|$, in a probabilistic sense, for random synthetic two-dimensional incompressible velocity fields $u(x,t)$ with given energy spectra.
We also propose an asymptotic correction factor to the BHT scaling arising from the time-dependence of $u$.
\end{abstract}

\maketitle


\section{Introduction and Setup}\label{s:intro}

We consider the evolution of a passive scalar $\tht(x,t)$ under
a prescribed velocity field $u(x,t)$ and source $g(x)$,
\begin{equation}\label{q:dtht}
   \dy_t\tht + u\cdot\gb\tht = \Delta\tht + g.
\end{equation}
For ``realistic'' $u$, arising from experiments or numerical simulations, the situation is quite complicated, with various regimes depending on $\|u\|$ and for different ranges of scales \cite{warhaft:00,holzer-siggia:94,shraiman-siggia:00}.
Confining ourselves in this paper to the zero Prandtl (or Schmidt) number limit, we assume that the energy spectrum (which is $\propto|u_k|^2|k|^{d-1}$) scales as $|k|^\beta$ for some $\beta<0$ for all sufficiently large $|k|$.
This leaves us with two asymptotic regimes, that of small $\|u\|$ and large $\|u\|$, relative to the unit diffusion coefficient.

When $\|u\|$ is large, the Batchelor--Corrsin--Obukhov theory \cite{batchelor:59,corrsin:51,obukhov:49} predicts that the tracer spectrum should scale as $|k|^{-(\beta+5)/2}$, for $|k|$ small enough that inertial effects dominate diffusion, giving $|k|^{-5/3}$ in 3d (with $\beta=-\frac53$) and $|k|^{-1}$ in 2d (with $\beta=-3$).
This regime is often called the inertial--convective range \cite[p.~367]{vallis:aofd}.
Using a random velocity field that is white noise in time, Kraichnan \cite{kraichnan:68,kimura-kraichnan:93} derived these tracer spectra using a closure scheme called Lagrangian history direct interaction approximation.  In \cite{jolly-dw:ttur} we assumed a tracer spectrum (based on the Batchelor--Corrsin--Obukhov  theory) in order to estimate the extent of the tracer cascade.

In the small $\|u\|$ case, or for sufficiently large $|k|$ for any $u$, also known as the inertial--diffusive range \cite[p.~369]{vallis:aofd}, Batchelor, Howells and Townsend \cite{batchelor-howells-townsend:59} predicted in 1959 that the tracer spectrum should scale as $|k|^{\beta-4}$.
In contrast to Kraichnan's argument, their theory implicitly assumes that the velocity does not vary much relative to some unspecified tracer relaxation time.
The resulting tracer spectrum is thus independent of the evolution of $u$, which some find unsatisfactory, particularly in two dimensions where a static velocity field gives rise to an integrable dynamical system and thus no ``turbulence''; for this and related issues, see, e.g., \cite{ngan-pierrehumbert:00}.

In this paper we revisit the latter (henceforth BHT) regime, where diffusion predominates over advection at all scales.
For conceptual simplicity (and to avoid unsettled questions over the correct NSE velocity spectra), we use a synthetic random velocity field whose power spectrum scales as $|k|^\beta$, in both the static and time varying cases, with the aim of proving rigorously under what conditions one may expect to recover the BHT spectrum.
We hope that a complete understanding of one regime of ``tracer turbulence'' will bring us a step closer to the full picture, which likely lies at the intersection of multiple regimes.

Mathematically, one can prove an upper bound for the tracer spectrum that obeys the BHT scaling albeit with a worse ``constant'', as done below in the proof of Lemma~\ref{t:thtn}.
The lower bound however, is a different story: it is quite easy to contrive velocities and tracer sources that, despite having the ``correct'' spectra, do not give rise to the BHT tracer spectrum.
To wit, let $g(x)$ be fixed and $u(x,t)=\sgb\psi(x,t)$ where $\sgb\psi\cdot\gb g=0$; e.g., one can take $\psi(\cdot,t)=F(\Delta^{-1}g;t)$ for any $F(\cdot;t)\in C^1$.
Then one can verify that $\theta(x)=-(\Delta^{-1}g)(x)$ is the solution of \eqref{q:dtht}, which is independent of $u(x,t)$ as long as the latter preserves level sets of $g$.
There are thus (infinitely) many possible spectra of $u$ compatible with a given spectrum of $\theta$.
To avoid such pathologies without having to enumerate them, we adopt a probabilistic approach, where one seeks to prove that the tracer spectrum satisfies the BHT scaling with probability approaching one (in a well-defined space) as $|k|\to\infty$.

Numerical simulations of \eqref{q:dtht} can be found, e.g., in \cite{chasnov-canuto-rogallo:88,chasnov:91,gotoh-watanabe-suzuki:11,yeung-sreenivasan:13,yeung-sreenivasan:14}, all done in three dimensions with the aim of comparing with experimental data.
Two-dimensional simulations have been done by \cite{holzer-siggia:94}, who employed a synthetic Ornstein--Uhlenbeck velocity, and \cite{chen-kraichnan:98}, who employed a synthetic pseudo-white-noise velocity to model the theory of \cite{kraichnan:68};
the mathematical result of this paper covers (and is consistent  with) the first of these studies, but does not formally cover the second study due to the time-dependence of the velocity (although the spectra again agree).
As is clear below, in this paper we calculate the mean but only bound the variance of the tracer spectrum, so our result says nothing about, e.g., possible non-Gaussian behaviour (to begin with, it is not clear if our synthetic velocity is Gaussian).

\medskip\hbox to\hsize{\qquad\hrulefill\qquad}\medskip

For simplicity, we take $x\in\Dom:=[0,2\pi]^2$ and assume periodic boundary conditions in both directions.
The advecting velocity $u$ is taken to be incompressible, $\gb\cdot u=0$, with further regularity assumptions to be stated below as needed.
With no loss of generality, we assume that for all $t$
\begin{equation}
   \int_\Dom u(x,t) \;\mathrm{d}x = 0
   \quad\text{and}\quad
   \int_\Dom \tht(x,t) \;\mathrm{d}x = 0.
\end{equation}

We construct the velocity by putting
\begin{equation}\label{q:psi}
   u = -\dy_y\psi\boldsymbol{e}_x + \dy_x\psi\boldsymbol{e}_y =: \sgb\psi
   \quad\text{where}\quad
   \psi(x,t) = U\sump{k}\,|k|^{\beta}\ex^{\im(\phi_k^{}(t) + k\cdot x)}
\end{equation}
and $U$ is a positive constant parameter.
Here and elsewhere, $\sump{k}:=\sum_{k\in\Zahl^2\backslash\{0\}}$ excluding $k=0$.
Since $\psi(x,t)\in\Real$, its Fourier coefficients $\psi_k(t)=U|k|^\beta\ex^{\im\phi_k(t)}$ must satisfy $\psi_{-k}(t)=\bar\psi_{k}(t)$, which implies
\begin{equation}\label{q:rphase}
   \phi_{-k}(t)=-\phi_k(t).
\end{equation}
Except for this constraint, the phases $\phi_k(t)=2\pi-\phi_{-k}(t)\in\Unif(0,2\pi)$ are assumed to be independent (w.r.t.\ $k$) random variables uniformly distributed in $[0,2\pi)$.
Defining the Fourier upper half-plane
\begin{equation}\label{q:zu}
   \Zu:=\{(m,n):m\in\Zahl,\; n\in\mathbb{N}\} \cup \{(m,0):m\in\mathbb{N}\},
\end{equation}
we can say that the phases $\phi_k(t)$ are independent random variables for all $k\in\Zu$.
Analogously, by a (spectral) half-plane we mean this upper half-plane rigidly rotated by some angle: $\Zahl^2\cap\{k\in\Real^2: k\cdot a > 0 \text{ or } k\wedge a = |k|\,|a|\}$ for some fixed $a\in\Real^2-\{0\}$.

It is clear that \eqref{q:psi} gives rise to an energy spectrum that scales (up to integer lattice effects) as $|k|^{2\beta+3}$.
Denoting spectral projection by
\begin{equation}
   (\Proj_{\kappa,\kappa'}\tht)(x,t) := \sum_{\kappa\le|k|<\kappa'} \tht_k(t)\ex^{\im k\cdot x}
\end{equation}
(and its vector equivalent for $u$), this energy spectrum implies
\begin{equation}\label{q:ustatic}
   |\Proj_{\kappa,2\kappa} u|_{L^2}^2 \simeq \frac{2^{2\beta+4}-1}{2\beta+4}\,|D|\,U^2\kappa^{2\beta+4}
\end{equation}
where by $f(\kappa)\simeq g(\kappa)$ we mean that $\lim_{\kappa\to\infty} f(\kappa)/g(\kappa)=1$.
In this paper ``$\simeq$'' arises either from lattice effect, when we approximate sums over subsets of $\Zahl^2$ by the corresponding integrals over subsets of $\Real^2$, or from dropping terms of (relative) order $\kpg/\kappa$.
An estimate of the error in $\simeq$ can be found in the Appendix.
Unlike for the solution of, say, the 2d Navier--Stokes equations, we take this spectrum as extending all the way to infinity, and (for $\beta=-3$) may be thought of as the zero-viscosity limit of the spectrum of the 2d NSE.

For the tracer source $g$, we take for conceptual and technical simplicity
\begin{equation}\label{q:gkg}
   g(x) = -\tssum_k\,\gamma(|k|)|k|^2\ex^{\im(\xi_k^{} + k\cdot x)},\\
\end{equation}
where $g$ is bandwidth-limited, i.e.\ $\gamma(\kappa)=0$ for $\kappa\ge\kpg$;
for $g(x)$ to have zero mean, we also impose $\gamma(0)=0$.
We denote $\gamma_k:=\gamma(|k|)$.
Like $\phi_k$, we take the phases $\xi_k$ to be random in $\Unif(0,2\pi)$ and independent modulo the constraint $\xi_{-k}=-\xi_k$ [cf.~\eqref{q:rphase}].

\section{Static Velocity}\label{s:vstat}

We first consider the conceptually and technically simpler case where the velocity $u$ is independent of time.
Here $\theta$ is the solution of the elliptic equation
\begin{equation}\label{q:statht}
   -\Delta\tht + u\cdot\gb\tht = g
\end{equation}
whose existence, uniqueness and smoothness are well established provided that $u$ is sufficiently controlled.
However, these standard results do not give us the spectral shape, for which we need to estimate each spectral dyad $|\Pkk\tht|$ from above and below.

To this end, we write the solution $\tht$ of \eqref{q:statht} as the limit of the iteration
\begin{align}
  &\tht^{(0)} = -\Delta^{-1}g\\
  &\tht^{(n+1)} = -\Delta^{-1}\bigl[g - u\cdot\gb\tht^{(n)}\bigr],\label{q:thsn}
\end{align}
which is convergent if $|u|_{L^\infty(D)}^{}$, or equivalently the parameter $U$, is sufficiently small.
The precise criterion follows from the time-dependent case [\eqref{q:it0}--\eqref{q:idr} below], so we will not treat it explicitly here.
It will be convenient to define
\begin{equation}\label{q:vfi0}
   \vfi^{(n)} = -\Delta\tht^{(n)} - g
\end{equation}
with $\vfi^{(1)}=u\cdot\nabla(\Delta^{-1}g)$.

Notwithstanding the pathological example in the introduction, this setup does give the BHT spectrum with probability approaching 1 in the limit $\kappa\to\infty$.
More precisely, denoting expected value by $\EV$ and variance by $\Var$, both over i.i.d.\ random phases $\phi_k$ and $\xi_k$, we have the following:

\begin{theorem}\label{t:static}
Let the source $g$ be given by \eqref{q:gkg},
the velocity $u$ be given by \eqref{q:psi} with time-independent $\phi_k$, $\beta<-2$ and $U$ sufficiently small for convergence of the iteration \eqref{q:thsn},
and $\theta$ be the solution of \eqref{q:statht}.
Then we can write $\theta=-\Delta^{-1}g + \vtt + \delta\theta$, where $\vtt$ has the following properties when averaged over the random phases $\phi$ of the velocity $u$ and $\xi$ of the source $g$,
\begin{align}
   &\EV(|\Pkk\vtt|_{L^2}^2) \simeq \pi\,\mathcal{G}_0(g)\frac{4^{\beta}-1}{2\beta}\,U^2\kappa^{2\beta}\quad\text{and}\label{q:EVvtht}\\
   &\Var(|\Pkk\vtt|_{L^2}^2) \lesssim \frac{\pi\,\mathcal{G}_1(g)}{2}\frac{4^{2\beta-1}-1}{4\beta-2}\,U^4\kappa^{4\beta-2},\label{q:Varvtht}
\end{align}
with $\mathcal{G}_0$ and $\mathcal{G}_1$ given in Lemma~\ref{t:vfi0} below.
The remainder $\delta\tht$ is bounded as
\begin{equation}\label{q:bddtht}
   |\Pkk\delta\tht|_{L^2}^2 \le c_0(\beta,\kpg)|\gb^{-1}g|_{L^\infty}^2U^3\kappa^{2\beta}.
\end{equation}
\end{theorem}

That $\vtt$ has the correct BHT spectrum for large $\kappa$ follows from the fact that $\EV|\Pkk\vtt|_{L^2}^2/|\Pkk u|_{L^2}^2\propto\kappa^{-4}$.
That this happens with probability approaching 1 as $\kappa\to\infty$ follows from the fact that $\Var(|\Pkk\vtt|_{L^2}^2)^{1/2}/\EV(|\Pkk\vtt|_{L^2}^2)\propto\kappa^{-1}$.
The relative error $|\Pkk\delta\tht|_{L^2}^2/\EV|\Pkk\vtt|_{L^2}^2$ can be made vanishingly small {\em uniformly in $\kappa$\/} by taking $U|\gb^{-1}g|_{L^\infty}^2/|\gb^{-1}g|_{L^2}^2\to0$.
We note that there are no uncomputed constants in either \eqref{q:EVvtht} or \eqref{q:Varvtht}.

\begin{proof}
Equations \eqref{q:EVvtht}--\eqref{q:Varvtht} follow from Lemma~\ref{t:vfi0} with $\vtt=-\Delta^{-1}\vfi^{(1)}$.
The bound \eqref{q:bddtht} follows from Lemma~\ref{t:thtn}.
\end{proof}

\begin{lemma}\label{t:vfi0}
Let $u$ be given by \eqref{q:psi} and $g$ by \eqref{q:gkg}, with $\theta$ the solution of \eqref{q:dtht} and $\varphi^{(1)}$ as in~\eqref{q:vfi0}.
Then taking expected value and variance over i.i.d.\ random phases $\phi_k$ we have
\begin{align}
   &\EV(|\Pkk\vfi^{(1)}|_{L^2}^2) \simeq \pi\,\mathcal{G}_0(g)\frac{2^{2\beta+4}-1}{2\beta+4}\,U^2\kappa^{2\beta+4}\quad\text{and}\label{q:EVvfi}\\
   &\Var(|\Pkk\vfi^{(1)}|_{L^2}^2) \lesssim \frac{\pi\,\mathcal{G}_1(g)}{2}\frac{2^{4\beta+6}-1}{4\beta+6}\,U^4\kappa^{4\beta+6},\label{q:Varvfi}
\end{align}
where
\begin{align}
   &\mathcal{G}_0(g) = \tssum_j\,|j|^2\gamma_j^2
	= |\gb^{-1}g|_{L^2}^2 \quad\text{and}\\
   &\mathcal{G}_1(g) = \tssum_{j,k}\,(3j_x^2k_x^2 + j_x^2k_y^2 + j_y^2k_x^2 + 3j_y^2k_y^2 + 4i_xi_yj_xj_y)\,\gamma_j^2\gamma_k^2.
\end{align}
\end{lemma}

\begin{proof}
In Fourier space, with $j\wedge k:=j_xk_y-j_yk_x$, \eqref{q:vfi0} reads
\begin{equation}\label{q:vfi1}\begin{aligned}
  \vfi^{(1)}_k &= \tssum_j\,(j\wedge k)\,\psi_{k-j}g_j/|j|^2\\
	&= U\tssum_j\,(k\wedge j)|k-j|^{\beta}\gamma_j\ex^{\im(\phi_{k-j}+\xi_j)}.
\end{aligned}\end{equation}
For brevity, we write $\vfi:=\vfi^{(1)}$, and noting that $\vfi$ is linear in $U$, we put $U\equiv1$ throughout the proof.
Here and henceforth, unadorned norm $|\cdot|$ means $|\cdot|_{L^2(\Dom)}^{}$.
With this we compute
\begin{equation}\begin{aligned}
   \EV|\Pkk\vfi|_{L^2}^2
	&= \EV(\tssum_k\,|\vfi_k|^2)\\
	&= \tssum_k\,\EV|\vfi_k|^2
\end{aligned}\end{equation}
where all $k$ sums are understood to be over $\kappa\le|k|<2\kappa$.
We first rewrite
\begin{equation}
   \vfi_k = \sum_{j\in\Zu}(k\wedge j)\gamma_j\bigl(|k-j|^{\beta}\ex^{\im(\phi_{k-j}+\xi_j)} - |k+j|^{\beta}\ex^{\im(\phi_{k+j}-\xi_j)}\bigr).
\end{equation}
Writing $\cis\phi:=\ex^{\im\phi}$ for readability, we compute (there is no need to take $\EV$ over $\xi_k$ for this part of the computation)
\begin{equation}\label{q:aux31}\begin{aligned}
   \EV(\vfi_k\bar\vfi_k)
	&= \sum_{i,j\in\Zahl^2}\,(k\wedge i)(k\wedge j)\gamma_i\gamma_j\,|k-i|^{\beta}|k-j|^{\beta}\\ &\hbox to 99pt{}\ex^{\im(\xi_i-\xi_j)}\EV\,\cis(\phi_{k-i}-\phi_{k-j})\\
	&= \sum_{i,j\in\Zu}\,(k\wedge i)(k\wedge j)\gamma_i\gamma_j\\
	&\qquad\qquad\bigl\{\phantom{-} |k-i|^{\beta}|k-j|^{\beta}\ex^{\im(\xi_i-\xi_j)}\EV\,\cis(\phi_{k-i}-\phi_{k-j})\\
	&\qquad\qquad- |k+i|^{\beta}|k-j|^{\beta}\ex^{-\im(\xi_i+\xi_j)}\EV\,\cis(\phi_{k+i}-\phi_{k-j})\\
	&\qquad\qquad- |k-i|^{\beta}|k+j|^{\beta}\ex^{\im(\xi_i+\xi_j)}\EV\,\cis(\phi_{k-i}-\phi_{k+j})\\
	&\qquad\qquad+ |k+i|^{\beta}|k+j|^{\beta}\ex^{\im(\xi_j-\xi_i)}\EV\,\cis(\phi_{k+i}-\phi_{k+j})\bigr\}.
\end{aligned}\end{equation}
Assuming that $|k|>2\kpg$ and recalling that $\gamma_j=0$ when $|j|>\kpg$, all (relevant) wavevectors of the form $k+j$ will lie in the same half-plane in $\Zahl^2$, implying that, e.g., $\phi_{k-j}$ and $\phi_{k-i}$ are independent and thus $\EV\,\cis(\phi_{k-i}-\phi_{k-j})=0$ unless $i=j$.
This reasoning implies that the two middle terms vanish: for $\EV\,\cis(\phi_{k+i}-\phi_{k-j})$ to be non-zero, we must have $j=-i$, which is impossible since both $i$ and $j\in\Zu$; and similarly for $\EV\,\cis(\phi_{k-i}-\phi_{k+j})$.
Thus only the first and last terms contribute to the sum, and then only for $i=j$ when $\EV\,\cis(\cdots)=1$, giving us
\begin{equation}\label{q:vfi2}\begin{aligned}
   \EV|\vfi_k|^2 &= 2\sum\nolimits_{j\in\Zu} \gamma_j^2(k\wedge j)^2|k-j|^{2\beta}\\
	&= \sum\nolimits_{j\in\Zahl^2} \gamma_j^2(k\wedge j)^2|k-j|^{2\beta}
\end{aligned}\end{equation}

We now make two approximations, both of which have vanishing errors as $\kappa\to\infty$.
First, since $\kappa\gg\kpg$ by hypothesis, we approximate $|k-j|^{2\beta}\simeq|k|^{2\beta}$, bounding the error as follows.
We write
\begin{equation}
   |k-j|^2 = |k|^2\bigl|\hat k - j/|k|\bigr|^2
	= |k|^2\bigl(\bigl|\hat k - j_\parallel/|k|\bigr|^2 + |j_\bot|^2/|k|^2\bigr)
\end{equation}
where $\hat k:=k/|k|$, $j_\parallel:=(j\cdot k)\,\hat k$ and $j_\bot:=j-j_\parallel$.
Bounding from above, assuming that $|k|\ge 2\kpg\ge 2|j|$, we have
\begin{equation}
   |k-j|^2/|k|^2
	\le \bigl(1 + |j|/|k|\bigr)^2 + |j|^2/|k|^2
	\le 1 + 3|j|/|k|.
\end{equation}
While bounding from below, we find
\begin{equation}
   |k-j|^2/|k|^2
	\ge \bigl(1 - |j|/|k|\bigr)^2
	\ge 1 - 2|j|/|k|.
\end{equation}
These give us, for $|k|\ge3\kpg\ge3|j|$ and $\alpha>0$,
\begin{equation}\label{q:bdjkp}
   (1-2\kpg/|k|)^\alpha \le |k-j|^\alpha/|k|^\alpha \le (1+3\kpg/|k|)^\alpha
\end{equation}
and for $\alpha<0$
\begin{equation}\label{q:bdjk}
   (1+3\kpg/|k|)^\alpha \le |k-j|^\alpha/|k|^\alpha \le (1-2\kpg/|k|)^\alpha.
\end{equation}

For the second approximation, since $\kappa\gg1$, we can replace the sum over $\{k\in\Zahl^2:\kappa\le|k|<2\kappa\}$ by an integral over the annulus $\{(r,\varpi):\kappa\le r<2\kappa,\,0\le\varpi<2\pi\}$.
With these, we have
\begin{equation}\label{q:aux30}\begin{aligned}
   \EV(|\Pkk\vfi|_{L^2}^2)
	&= \sum_{\kappa\le|k|<2\kappa} \EV(|\vfi_k|^2)
	= \tssum_k\tssum_j\,\gamma_j^2(k\wedge j)^2|k-j|^{2\beta}\\
	&\simeq \tssum_k\tssum_j\,\gamma_j^2(k\wedge j)^2|k|^{2\beta}\\
	&\simeq \int_\kappa^{2\kappa}\!\!\int_0^{2\pi}\tssum_j\,r^2(j_x\sin\varpi-j_y\cos\varpi)^2\, r^{2\beta}\gamma_j^2 \;r\;\mathrm{d}\varpi\,\dr\\
	&= \int_\kappa^{2\kappa} \pi\tssum_j\,(j_x^2+j_y^2)\,r^{2\beta+3}\gamma_j^2 \dr,
\end{aligned}\end{equation}
which gives \eqref{q:EVvfi}.

For the variance, we compute
\begin{equation}\label{q:aux35}
   \Var\bigl(|\Pkk\vfi|_{L^2}^2\bigr)
	= \Var\bigl(\tssum_k\,|\vfi_k|^2\bigr)
	= \EV\bigl(\tssum_k\,|\vfi_k|^2\bigr)^2 - \Bigl(\EV\;\tssum_k\,|\vfi_k|^2\Bigr)^2.
\end{equation}
Having obtained the last term above, we write the first term as
\begin{equation}\label{q:aux36}
   \EV\bigl(\tssum_k\,|\vfi_k|^2\>\tssum_l\,|\vfi_l|^2\bigr)
	= \EV\;\tssum_k|\vfi_k|^4 + \EV\;\tssum_{k\ne l}|\vfi_k|^2|\vfi_l|^2.
\end{equation}
Similarly to what we did above, we compute $\EV\,\tssum_k|\vfi_k|^4 = \tssum_k\EV\,|\vfi_k|^4$ and, with $k$ fixed for now,
\begin{equation*}\begin{aligned}
   \EV\,|\vfi_k|^4 &= \EV(\vfi_k\bar\vfi_k\vfi_k\bar\vfi_k)\\
	&= \sum_{\pm;\,i,\,j,\,m,\,n\in\Zu} (\pm i\wedge k)(\pm j\wedge k)(\pm m\wedge k)(\pm n\wedge k) \gamma_i\gamma_j\gamma_m\gamma_n\\
	&\qquad|k\pm i|^\beta|k\pm j|^\beta|k\pm m|^\beta|k\pm n|^{\beta}\,\EV\,\ex^{\im(\xi_{\pm i}-\xi_{\pm j}+\xi_{\pm m}-\xi_{\pm n})}\\
	&\qquad\EV\,\cis(\phi_{k\pm i}-\phi_{k\pm j}+\phi_{k\pm m}-\phi_{k\pm n})
\end{aligned}\end{equation*}
where the $\pm$ sum is taken over the 16 combinations of $\pm i,\cdots,\pm n$.
Arguing as we did after \eqref{q:aux31}, since $|k|>2\kpg$, for $(-i,-j,-m,-n)$ we have $\EV\,\cis(\phi_{k-i}-\phi_{k-j}+\phi_{k-m}-\phi_{k-n})\ne0$ only when $i=j$, $m=n$ or $i=n$, $j=m$ (2 possible cases), and similarly for $\EV\,\cis(\phi_{k+i}-\phi_{k+j}+\phi_{k+m}-\phi_{k+n})$.
For $(-i,-j,+m,+n)$, we must have $i=j$, $m=n$ (only possible case); and similarly for $(+i,+j,-m,-n)$, $(-i,+j,+m,-n)$ and $(+i,-j,-m,+n)$.
In the remaining 12 combinations, the 8 with odd parities as well as $(+i,-j,+m,-n)$ and $(-i,+j,-m,+n)$, we have $\EV\,\cis(\cdots)=0$.
Relabelling indices as necessary, we thus find
\begin{equation*}\begin{aligned}
   \EV\,|\vfi_k|^4
	&= 2 \sum_{i,\,j\in\Zu} (k\wedge i)^2(k\wedge j)^2\gamma_i^2\gamma_j^2\\ &\qquad\qquad(|k-i|^{2\beta}|k-j|^{2\beta} +|k+i|^{2\beta}|k+j|^{2\beta} + 2|k-i|^{2\beta}|k+j|^{2\beta})\\
	&= 2 \Bigl(\sum_{j\in\Zahl^2} \gamma_j^2(k\wedge j)^2|k-j|^{2\beta}\Bigr)^2 = 2\,(\EV\,|\vfi_k|^2)^2.
\end{aligned}\end{equation*}

Next we consider, for $k\ne l$,
\begin{equation*}\begin{aligned}
   \EV\,|\vfi_k|^2|\vfi_l|^2 &= \EV(\vfi_k\bar\vfi_k\vfi_l\bar\vfi_l)\\
	&= \sum_{\pm;\,i,\,j,\,m,\,n\in\Zu} (\pm i\wedge k)(\pm j\wedge k)(\pm m\wedge l)(\pm n\wedge l) \gamma_i\gamma_j\gamma_m\gamma_n\\
	&\qquad|k\pm i|^\beta|k\pm j|^\beta|l\pm m|^\beta|l\pm n|^{\beta}\,\EV\,\ex^{\im(\xi_{\pm i}-\xi_{\pm j}+\xi_{\pm m}-\xi_{\pm n})}\\
	&\qquad\EV\,\cis(\phi_{k\pm i}-\phi_{k\pm j}+\phi_{l\pm m}-\phi_{l\pm n}).
\end{aligned}\end{equation*}
For $(+i,+j,+m,+n)$, $(-i,-j,-m,-n)$, $(-i,-j,+m,+n)$ and $(+i,+j,-m,-n)$, only, we have contributions in the ``straight'' case $i=j$, $m=n$, in which case $\ex^{\im(\xi_{\pm i}-\xi_{\pm j}+\xi_{\pm m}-\xi_{\pm n})}=1$ (valid regardless of whether $\xi$ is random).
These combinations contribute
\begin{equation}
   S_1 = \sum_{j,n\in\Zahl^2}\,(k\wedge j)^2(l\wedge n)^2\gamma_j^2\gamma_n^2|k-j|^{2\beta}|l-n|^{2\beta}.
\end{equation}
Writing $d:=k-l$, we also have contributions in 9 ``cross'' cases: e.g., for $(+i,+j,+m,+n)$ we have $i=n-d$, $m=j+d$, which (unlike the straight cases) has a phase $\EV\,\ex^{\im(\xi_{n-d}-\xi_j+\xi_{j+d}-\xi_n)}$ which is non-zero if and only if $n=j-d$.
These combinations contribute
\begin{equation}
   S_2 = \sum_{j\in\Zahl^2}\,(k\wedge j)^2(l\wedge n)^2\gamma_j^2\gamma_n^2|k-j|^{2\beta}|l-n|^{2\beta}
	\qquad\text{with }n=j-d.
\end{equation}
Since the summands in both $S_1$ and $S_2$ are non-negative, we have $S_2\le S_1$.

Putting things together, we have
\begin{equation}\begin{aligned}
   \EV\,|\vfi_k|^2|\vfi_l|^2 &\le 2\sum_{j,n\in\Zahl^2}\,(k\wedge j)^2(l\wedge n)^2\gamma_j^2\gamma_n^2|k+j|^{2\beta}|l+n|^{2\beta}\\
	&= 2\sum_{j\in\Zahl^2} \gamma_j^2(k\wedge j)^2|k-j|^{2\beta}\>\sum_{n\in\Zahl^2} \gamma_n^2(l\wedge n)^2|l-n|^{2\beta}.
\end{aligned}\end{equation}

Using this in \eqref{q:aux35}--\eqref{q:aux36}, we have
\begin{equation}\begin{aligned}
   \Var(|\Pkk\vfi|_{L^2}^2) &= \tssum_k\,\bigl(\EV\,|\vfi_k|^2\bigr)^2\\
	&\le \tssum_k\,\bigl(\tssum_j\,{(j\wedge k)^2|k-j|^{2\beta}}\gamma_j^2\bigr)^2\\
	&\lesssim \int_\kappa^{2\kappa} r^{4\beta+5} \int_0^{2\pi} \bigl(\tssum_j\,(j_x\sin\varpi - j_y\cos\varpi)^2\gamma_j^2\bigr)^2 \;\mathrm{d}\varpi\dr
\end{aligned}\end{equation}
where at the last step we have again made the two approximations that led to \eqref{q:aux30}.
Computing the $\varpi$ integral explicitly,
\begin{align*}
   \int_0^{2\pi}\bigl(\tssum_j\cdots\bigr)^2\;\mathrm{d}\varpi
	&= \tssum_{ij}\int_0^{2\pi}(i_x\sin\varpi-i_y\cos\varpi)^2(j_x\sin\varpi-j_y\cos\varpi)^2 \gamma_i^2\gamma_j^2\;\mathrm{d}\varpi\\
	&= \frac{\pi}{4}\tssum_{ij}\, (3i_x^2j_x^2 + i_x^2j_y^2 + i_y^2j_x^2 + 3i_y^2j_y^2 + 4i_xi_yj_xj_y) \gamma_i^2\gamma_j^2,
\end{align*}
which gives \eqref{q:Varvfi}.
\end{proof}

Next, we show that, for sufficiently small $U$, $\tht$ is dominated at all large scales by its leading-order approximation $\tht^{(1)}$, i.e.\ the iteration~\eqref{q:thsn} does not make $\tht^{(n)}$ much ``worse''.

\begin{lemma}\label{t:thtn}
Let $\beta<-2$ and $U$ satisfy \eqref{q:cond2} below.
Then we have for all $n\ge1$
\begin{equation}
   |\tht^{(n)}_k-\tht^{(1)}_k| \le U^{3/2}|\gb^{-1}g|_{L^\infty}^{}|k|^{-2}K_\beta(k),
\end{equation}
where $K_\beta(k)$ is defined shortly below.
\end{lemma}

\begin{proof}
We first derive a bound for $\delta\tht^{(1)}_k=\vtt_k$.
From \eqref{q:vfi1}, we have
\begin{equation}
   |\vtt_k|
	\le U|k|^{-2} \tssum_j\,|k|\,|j|\gamma_j|k-j|^\beta
\end{equation}
For $|k|>2\kpg$, we have $|k-j|>\sfrac12|k|$ and thus $|k-j|^\beta<2^{-\beta}|k|^\beta$; this gives
\begin{equation}
   |\vtt_k|
	\le 2^{-\beta}|k|^{\beta-1}U\tssum_j\,|j|\gamma_j
	= 2^{-\beta}|k|^{\beta-1}U\,|\gb^{-1}g|_{L^\infty}^{}.
\end{equation}
For small $|k|\le2\kpg$, we simply bound $|j-k|^\beta\le1$ since $|j-k|\ge1$ and obtain
\begin{equation}
   |\vtt_k|
	\le U|k|^{-1} \tssum_j\,|j|\gamma_j
	\le 2U|k|^{-1}|\gb^{-1}g|_{L^\infty}^{}.
\end{equation}
For what follows, we adopt the slightly worse bound
\begin{equation}\label{q:bdvfi1}
   |\vtt_k|
	\le |\gb^{-1}g|_{L^\infty}^{}U \Gamma_k
\end{equation}
where $\Gamma_k=\Gamma(|k|;\beta)=|k|^{-2}\min\{2\kpg,(2\kpg)^{-\beta}|k|^{\beta+1}\}$ is monotone decreasing in $|k|$ and satisfies
\begin{equation}\label{q:kbeta}
   \Gamma(sk;\beta)\le s^{\beta-1}\Gamma(k;\beta)
	\qquad\text{for all }s\in(0,1).
\end{equation}
We point out that, using this to bound $|\Pkk\vtt|_{L^2}^2$, the result will scale as $\kappa^{2\beta}$, i.e.\ it has the same scaling as the expected value \eqref{q:EVvfi} though obviously with a larger multiplier (among other things, depending on $|\gb^{-1}g|_{L^\infty}^2$ instead of $|\gb^{-1}g|_{L^2}^2$).

To complete the proof, we will need to bound sums of the form
\begin{equation}
   S = \sum\nolimits_{j\ne k}\,(k\wedge j)|k-j|^\beta\Gamma_j\,.
\end{equation}
When $|k|\le2\kpg$, we simply estimate
\begin{equation}\label{q:bdsumj0}
   \tssum_j\,(k\wedge j)|k-j|^\beta\Gamma_j
	\le 2\kpg\,\tssum_j\,|j|\Gamma_j
	\le 2\kpg/|\beta+2|.
\end{equation}
provided that $\beta<-2$.

Next we consider the case $|k|>2\kpg$.
Fixing $\eta=\sfrac1{10}$, we split the sum into four parts.
For $|j|\ge(1+\eta)|k|$, we have $|k-j|\ge\eta|k|$ and thus $|k-j|^\beta\le\eta^\beta|k|^\beta$, which gives
\begin{equation}
   (k\wedge j)|k-j|^\beta\Gamma_j \le \eta^\beta|k|^{\beta+1}\Gamma_j|j|.
\end{equation}
For $|j|\le(1-\eta)|k|$, we have $|k-j|\ge\eta|k|$ and recover the above bound.
For $(1-\eta)|k|\le|j|\le(1+\eta)|k|$ and, with $m:=j-k$, $|m|\ge\eta|k|$, we have again
\begin{equation}
   (k\wedge j)|k-j|^\beta\Gamma_j
	= (k\wedge m)|m|^\beta\Gamma_j
	\le \eta^\beta|k|^{\beta+1}\Gamma_j|j|.
\end{equation}
Finally, for $(1-\eta)|k|\le|j|\le(1+\eta)|k|$ and $|m|\le\eta|k|$, we have
\begin{equation}\begin{aligned}
   (k\wedge j)|k-j|^\beta\Gamma_j
	&= (k\wedge m)|m|^\beta\Gamma_j\\
	&\le |m|^{\beta+1}|k|\Gamma_{(1-\eta)k}
	\le |m|^{\beta+1}|k|(1-\eta)^{\beta+1}\Gamma_k.
\end{aligned}\end{equation}
We therefore find
\begin{equation}\label{q:aux0}
   \Bigl|\sum_{j\ne k} (k\wedge j)|k-j|^\beta\Gamma_j\Bigr|
	\le \eta^\beta|k|^{\beta+1}\sum_{j\ne0}\,|j|\Gamma_j
	+ (1-\eta)^{\beta+1}|k|\Gamma_{k}\sum_{|m|\le\eta|k|}\,|m|^{\beta+1}.
\end{equation}
Computing the sums, writing $\lesssim$ for $\le$ up to lattice effects, we find
\begin{equation}\begin{aligned}
   \tssum_j\,|j|\Gamma_j
	&\le 2\kpg\tssum_{|j|\le2\kpg} |j|^{-1} + \tssum_{|j|>2\kpg} (2\kpg)^{-\beta}|j|^{\beta-1}\\
	&\lesssim 2\pi(2\kpg)^2 + 2\pi \frac{2\kpg}{|\beta+1|}
\end{aligned}\end{equation}
and $\sum_{|m|\le\eta|k|} |m|^{\beta+1}\lesssim 2\pi M_\beta(k;\eta)$ where
\begin{equation}\label{q:aux2}
   M_\beta(k;\eta) = \begin{cases} (\eta|k|)^{\beta+3}/(\beta+3) &\text{for }\beta>-3,\\
		\log(\eta|k|) &\text{for }\beta=-3,\\
		1/|\beta+3| &\text{for }\beta<-3. \end{cases}
\end{equation}
Putting together \eqref{q:aux0}--\eqref{q:aux2}, we find
\begin{equation}\label{q:bdsumj9}
   \Bigl|\sum_{j\ne k} (k\wedge j)|k-j|^\beta\Gamma_j\Bigr|
	\le c_1(\beta,\kpg,\eta)|k|^{\beta+1} + c_2(\beta,\eta)M_\beta(k,\eta)|k|\,\Gamma_k.
\end{equation}
For the rhs to be subdominant to $|k|^2\Gamma_k$ for large $|k|$, we must require $\beta<-2$.

Now let $\delta\tht^{(n)}:=\tht^{(n)}-\tht^{(n-1)}$
and rewrite \eqref{q:bdvfi1} as
\begin{equation}
   |\delta\tht^{(1)}_k| \le |\gb^{-1}g|_{L^\infty}^{}U\Gamma_k.
\end{equation}
Now we have from \eqref{q:thsn}
\begin{equation}
   \delta\tht^{(n+1)} = -A^{-1}(u\cdot\gb\delta\tht^{(n)}).
\end{equation}
We thus have, for each $k$,
\begin{equation}
   \delta\tht^{(n+1)}_k = U|k|^{-2}\sum\nolimits_{j\ne k}\,(k\wedge j)|k-j|^\beta\delta\tht^{(n)}_j.
\end{equation}
Using \eqref{q:bdsumj0} and \eqref{q:bdsumj9}, we have if $|\delta\tht^{(n)}_k|\le d_n\Gamma_k$,
\begin{equation}
   |\delta\tht^{(n+1)}_k| \le U\,c_4(\beta,\kpg) d_n\Gamma_k
\end{equation}
where $c_4$ is independent of $n$.
Therefore, taking $U$ small enough so that
\begin{equation}\label{q:cond2}
   U^{1/2}c_4(\beta,\kpg) \le \sfrac12
\end{equation}
the differences are bounded as
\begin{equation}
   |\delta\tht^{(n)}_k|\le 2^{-n+1}U^{1/2}|\gb^{-1}g|_{L^\infty}^{}U\Gamma_k.
\end{equation}
We thus have
\begin{align}
   |\tht^{(n)}_k-\tht^{(1)}_k|
	&\le |\delta\tht^{(n)}_k| + |\delta\tht^{(n-1)}_k| + \cdots + |\delta\tht^{(2)}_k|\notag\\
	&\le U^{3/2}|\gb^{-1}g|_{L^\infty}^{}\Gamma_k
\end{align}
uniformly in $n$.
\end{proof}

\section{Time-dependent Velocity}\label{s:vt}

We now turn to the more interesting case where $u$ depends on time.
Analogously to \eqref{q:psi}, here we take
\begin{equation}\label{q:psit}
   u = -\dy_y\psi\boldsymbol{e}_x + \dy_x\psi\boldsymbol{e}_y =: \sgb\psi
   \quad\text{where}\quad
   \psi(x,t) = U\sump{k}\,|k|^{\beta}\ex^{\im(\phi_k^{}(t) + k\cdot x)}
\end{equation}
for a positive real constant $U$.
The phases $\phi_k(t)$ are independent stationary random processes, with the proviso that $\phi_{-k}(t)=-\phi_k(t)$, satisfying
\begin{equation}
   \cov(\ex^{\im\phi_j(s)},\ex^{\im\phi_k(r)})
	= \EV\,\ex^{\im\phi_j(s)-\im\phi_k(r)}
	= \delta_{jk}\Phi_k(s-r).
\end{equation}
We assume that the time correlation function is of the form $\Phi_k(t)=\Phi(\chi_k^{}|t|)$ with $\Phi\in C^2(\Real_+;[-1,1])$ and $\Phi(0)=1$.
The correlation timescale $\chi_k^{-1}$ is assumed not to grow too rapidly with $|k|$,
\begin{equation}\label{q:chik}
   \lim_{|k|\to\infty}\chi_k^{}|k|^{-\alpha} = 0
	\qquad\text{for all }\alpha\ge2.
\end{equation}
We also assume that $\phi_k(t)$ has sufficient smoothness in $t$ for the usual Riemann integral to be defined.
As noted in the introduction, the regime considered here is the opposite of Kraichnan's white noise velocity \cite{kraichnan:68}.

Confirming the original intuition of BHT, at least for our particular model, the correlation shape function $\Phi$ and the correlation timescale only affect the tracer spectrum for smaller $|k|$, and we recover the static case of the last section as $|k|\to\infty$ {\em independently\/} of $\Phi$ and $\chi_k^{}$:

\begin{theorem}\label{t:vt}
Let the source $g(x)$ be given by \eqref{q:gkg} and $\tht(x,t)$ be the solution of \eqref{q:dtht}.
Let the velocity $u(x,t)$ be given by \eqref{q:psit}--\eqref{q:chik} with $\beta<-2$;
assume also that $U$ is small enough that the convergence condition \eqref{q:idr} below holds.
Then we can write $\tht=-\Delta^{-1}g + \vtt + \delta\tht$, where $\vtt(x,t)$ satisfies the following probabilistic estimate over the random phases $\phi$ of the velocity $u$,
\begin{equation}\label{q:EVvttk}\begin{aligned}
   \lim_{t\to\infty}&\EV|\vtt_k(t)|^2 =
	U^2\sum\nolimits_j\,\gamma_j^2\frac{(k\wedge j)^2|k-j|^{2\beta}}{|k|^4} \times\\
	&\qquad\Bigl[ 1 + \frac{\chi_{k-j}^{}}{|k|^2}\Phi'(0) + \cdots + \frac{\chi_{k-j}^{n}}{|k|^{2n}}\int_0^\infty \ex^{-s|k|^2/\chi_{k-j}^{}}\Phi^{(n)}(s) \ds\Bigr].
\end{aligned}\end{equation}
When $\kappa^2/\chi_k^{}\gg1$, we recover the static spectrum
\begin{align}
   &\lim_{t\to\infty}\EV(|\Pkk\vtt(t)|_{L^2}^2) \simeq \pi\,\mathcal{G}_0(g)\frac{4^\beta-1}{2\beta}\,U^2\kappa^{2\beta}\quad\text{and}\label{q:EVvtt}\\
   &\lim_{t\to\infty}\Var(|\Pkk\vtt(t)|_{L^2}^2) \lesssim \frac{\pi\,\mathcal{G}_1(g)}{2}\frac{4^{2\beta-1}-1}{4\beta-2}\,U^4\kappa^{4\beta-2},\label{q:Varvtt}
\end{align}
with $\mathcal{G}_0$ and $\mathcal{G}_1$ as in Lemma~\ref{t:vfi0}.
The remainder $\delta\tht(x,t)$ is bounded for all $t\ge0$ as
\begin{equation}\label{q:bddthtt}
   |\Pkk\delta\tht(t)|_{L^2}^2 \le c_0(\beta,\kpg)|\gb^{-1}g|_{L^\infty}^2U^3\kappa^{2\beta}.
\end{equation}
\end{theorem}

\noindent In \eqref{q:EVvtt}--\eqref{q:Varvtt}, the $\simeq$ is to be regarded as up to remainders of order $\chi_k^{}/\kappa^2$ as well as $\kpg/\kappa$ and lattice effects.

We note that the ``static'' behaviour in \eqref{q:EVvtt}--\eqref{q:Varvtt} can be regarded as arising from the first term (the 1) in the bracket in \eqref{q:EVvttk}.
For some forms of $\chi_{k-j}^{}$ and sufficiently smooth $\Phi$, the asymptotic expression in \eqref{q:EVvttk} can be summed to give expressions analogous to \eqref{q:EVvtt} that may be regarded as higher-order corrections to the BHT spectrum.
For example, with $\chi_{k-j}^{}=\chi\,|k-j|^\eta$ for $\eta\ge0$, we have
\begin{equation}\label{q:bhtcorr}\begin{aligned}
   \lim_{t\to\infty}\EV|\Pkk\vtt(t)|_{L^2}^2
	&\simeq \pi\,\mathcal{G}_0(g) U^2\kappa^{2\beta}
	\Bigl[ \frac{4^\beta-1}{2\beta} + \frac{2^{2\beta+\eta-2}-1}{2\beta+\eta-2} \Phi'(0)\chi\kappa^{\eta-2}\\ &\qquad+ \cdots + \frac{2^{2\beta+n(\eta-2)}-1}{2\beta + n(\eta-2)}\Phi^{(n)}(0)\chi^{n}\kappa^{(\eta-2)n} + \cdots \Bigr].
\end{aligned}\end{equation}
As is usually the case with asymptotic series, this expansion is not convergent and must be truncated at some ($\kappa$-dependent) order for optimal accuracy.
We note that the later terms of this asymptotic expansion may be dominated by our bound for the lattice effects;
the latter may be sub-optimal (even though \eqref{q:bdjkp}--\eqref{q:bdjk} are qualitatively sharp, their uses are not), however, in which case some terms of the asymptotic correction \eqref{q:bhtcorr} may manifest themselves.


\begin{proof}
As in the static case, we write the solution $\tht(x,t)$ of \eqref{q:dtht} as the limit of iterates $\tht^{(n)}(x,t)$ defined by
\begin{align}
   &\tht^{(0)} = -\Delta^{-1}g,\label{q:it0}\\
   &\tht^{(n+1)}(\cdot,t) = -\Delta^{-1}g - \int_0^t \ex^{(t-s)\Delta}[u(\cdot,s)\cdot\gb\tht^{(n)}(\cdot,s)] \ds.\label{q:itn}
\end{align}
Here $\ex^{-t\Delta}$ is the heat kernel, i.e.\ $\tht^{(n+1)}$ is the solution of
\begin{equation}
   (\dy_t - \Delta)\tht^{(n+1)} = g - u\cdot\gb\tht^{(n)}
   \quad\text{with}\quad \tht^{(n+1)}(\cdot,0) = -\Delta^{-1}g.
\end{equation}

Considering \eqref{q:itn} as a mapping $T:\tht^{(n)}\mapsto\tht^{(n+1)}$, convergence of the iterations \eqref{q:it0}--\eqref{q:itn} would follow from the contractivity of $T$.
To prove the latter, we write $\delta\tht^{(n)}:=\tht^{(n)}-\tht^{(n-1)}$ and observe that it satisfies
\begin{equation}
   (\dy_t - \Delta)\delta\tht^{(n)} = -u\cdot\gb\delta\tht^{(n-1)}
   \quad\text{with}\quad \delta\tht^{(n)}(\cdot,0) = 0.
\end{equation}
Multiplying this by $\delta\tht^{(n)}$ in $L^2(\Dom)$, we find
\begin{equation}\begin{aligned}
   \frac12\ddt{\;}|\delta\tht^{(n)}|^2 + |\gb\delta\tht^{(n)}|^2
	&= -(\gb\cdot(u\,\delta\tht^{(n-1)}),\delta\tht^{(n)})\\
	&\le \frac12\,|\gb\delta\tht^{(n)}|^2 + c\,|u|_{L^\infty}^2|\delta\tht^{(n-1)}|^2.
\end{aligned}\end{equation}
Integrating in time, we find
\begin{equation}\begin{aligned}
   |\delta\tht^{(n)}(t)|^2 &+ \int_0^t |\gb\delta\tht^{(n)}(s)|^2 \ds
	\le c_1\,|u|_{L^\infty([0,t],L^\infty(\Dom))}^2 \int_0^t |\delta\tht^{(n-1)}(s)|^2 \ds\\
	&\le c_1\,|u|_{L^\infty([0,t],L^\infty(\Dom))}^2 \int_0^t |\gb\delta\tht^{(n-1)}(s)|^2 \ds,
\end{aligned}\end{equation}
so convergence of $\tht^{(n)}$ in $L^2([0,t],H^1(\Dom))$ would follow from
\begin{equation}\label{q:idr}
   \alpha:=c_1\,|u|_{L^\infty([0,t],L^\infty(\Dom))}^2 < 1.
\end{equation}

We now turn our attention to  $\vtt$, given by
\begin{equation}
   \vtt(t) = \theta^{(1)}(t) + \Delta^{-1}g
	= \int_0^t \ex^{(t-s)\Delta}u(s)\cdot\gb\Delta^{-1}g \ds,
\end{equation}
and whose Fourier coefficients satisfy
\begin{equation}
   \vtt_k(t) = \sum\nolimits_j\,\gamma_j\ex^{\im\xi_j}(k\wedge j)|k-j|^\beta \int_0^t \ex^{(s-t)|k|^2+\im\phi_{k-j}(s)} \ds.
\end{equation}
To keep the presentation manageable, we shall ignore the fact that $\phi_{-k}=-\phi_k$ in what follows, treating the velocity and source as complex-valued.
The real-valued case can be done following the computation in the previous section, giving identical result as obtained below.
We compute
\begin{align}
   \EV\,\vtt_k(t)\overline{\vtt_k(t)}
	&= \sum\nolimits_{ij} \gamma_j\gamma_i\ex^{\im\xi_j-\im\xi_i}(k\wedge j)(k\wedge i)|k-j|^\beta|k-i|^\beta\notag\\
	&\qquad\qquad \EV\Bigl\{ \int_0^t \ex^{(s-t)|k|^2+\im\phi_{k-j}(s)} \ds \int_0^t \ex^{(r-t)|k|^2-\im\phi_{k-i}(r)} \dr\Bigr\}\notag\\
	&= \sum\nolimits_{ij} (\cdots) \int_0^t\int_0^t \ex^{(s+r-2t)|k|^2}\EV\,\ex^{\im\phi_{k-j}(s)-\im\phi_{k-i}(r)} \ds\dr\notag\\
	&= \sum\nolimits_j \gamma_j^2(k\wedge j)^2|k-j|^{2\beta}\int_0^t\int_0^t \ex^{(s+r-2t)|k|^2}\Phi_{k-j}(s-r) \ds\dr.\label{q:Evttk}
\end{align}
As a check, putting $\Phi_{k}\equiv1$ gives us the static solution as $t\to\infty$:
\begin{equation}\label{q:intstat}
   \int_0^t\int_0^t \ex^{(s+r-2t)|k|^2} \ds\dr
	= \frac{\bigl(1-\ex^{-t|k|^2}\bigr)^2}{|k|^4}
	\to |k|^{-4} \qquad\text{as }t\to\infty.
\end{equation}

For convenience, we define
\begin{equation}
   \iPhi_k(t) := 2 \int_{-t}^t \Phi_k(2s) \ds;
\end{equation}
with $\Phi_k(s)=\Phi(\chi_k^{}|s|)$ this gives
\begin{equation}\label{q:iphik}
   \iPhi_k(t) = 2\chi_k^{-1}\int_{0}^{\chi_k^{}t} \Phi(2s) \ds
	=: \chi_k^{-1} \iPhi(\chi_k^{}t).
\end{equation}
We now rewrite the integral in \eqref{q:Evttk} in terms of $\tau=\frac12(s+r)$ and $\sigma=\frac12(s-r)$,
\begin{align}
   \int_0^t\int_0^t &\ex^{(s+r-2t)|k|^2}\Phi_{k-j}(s-r) \ds\dr\label{q:intst}\\
	&= 2\int_0^{t/2} \ex^{2|k|^2(\tau-t)} \int_{-\tau}^\tau \Phi_{k-j}(2\sigma) \dsig\dtau
	+ 2\int_{t/2}^t \ex^{2|k|^2(\tau-t)} \int_{\tau-t}^{t-\tau} \Phi_{k-j}(2\sigma) \dsig\dtau\notag\\
	&= 2\int_0^{t/2} \ex^{2|k|^2(\tau-t)} \iPhi_{k-j}(\tau) \dtau + 2\int_{t/2}^t \ex^{2|k|^2(\tau-t)} \iPhi_{k-j}(t-\tau) \dtau.\notag
\end{align}
As $t\to\infty$, the first integral will be vanishingly small, viz.,
\begin{equation}\begin{aligned}
   \Bigl|\int_0^{t/2} \ex^{2|k|^2(\tau-t)} \iPhi_{k-j}(\tau)\dtau\Bigr|
	&\le \int_0^{t/2} 4\tau\,\ex^{2|k|^2(\tau-t)}\dtau\\
	&= |k|^{-4}\ex^{-|k|^2t}\bigl(|k|^2t - 1 + \ex^{-|k|^2t}\bigr).
\end{aligned}\end{equation}
For any $\eps>0$, the bound on the rhs will be ${}\le \eps|k|^{-4}$ for all sufficiently large $t$.
The contribution to the integral in \eqref{q:intst} comes almost entirely from the second integral, and then only when $t-\tau$ is small.

We write the second integral in \eqref{q:intst} as
\begin{align}
   2\int_{t/2}^t \ex^{2|k|^2(\tau-t)} &\iPhi_{k-j}(t-\tau) \dtau
	=  2\int_0^{t/2} \ex^{-2|k|^2\tau} \iPhi_{k-j}(\tau) \dtau & &\notag\\
	&= \frac2{\chi_{k-j}^{}} \int_0^{t/2} \ex^{-2|k|^2\tau} \iPhi(\chi_{k-j}^{}\tau) \dtau & &\text{by \eqref{q:iphik}}\notag\\
	&\to \frac2{\chi_{k-j}^{}} \int_0^\infty \ex^{-2|k|^2\tau} \iPhi(\chi_{k-j}^{}\tau) \dtau & &\text{as }t\to\infty\notag\\
	&= \frac2{\chi_{k-j}^2} \int_0^\infty \ex^{-2\tau'|k|^2/\chi_{k-j}^{}} \iPhi(\tau') \dtau' & & \notag\\
	&= \frac1{\chi_{k-j}^{}|k|^2} \int_0^\infty \ex^{-\tau|k|^2/\chi_{k-j}^{}} \Phi(\tau) \dtau & &
\end{align}
where for the last equality we have integrated by parts using the fact that $\iPhi(0)=0$ and $\ex^{\tau\cdots}\iPhi(\tau)\to0$ as $\tau\to\infty$, and changed variables again to remove a factor of~2.
To obtain the large $|k|$ behaviour, we integrate by parts again using the fact that $\Phi(0)=1$,
\begin{equation}\label{q:aux37}\begin{aligned}
   \frac1{\chi_{k-j}^{}|k|^2} \int_0^\infty &\ex^{-\tau|k|^2/\chi_{k-j}^{}} \Phi(\tau) \dtau\\
	&= \frac1{|k|^4} + \frac1{|k|^4} \int_0^\infty \ex^{-\tau|k|^2/\chi_{k-j}^{}} \Phi'(\tau) \dtau\\
	&= \frac1{|k|^4} + \frac{\chi_{k-j}^{}}{|k|^6}\Phi'(0) + \frac{\chi_{k-j}}{|k|^6} \int_0^\infty \ex^{-\tau|k|^2/\chi_{k-j}^{}} \Phi''(\tau) \dtau.
\end{aligned}\end{equation}
Since $\Phi\in C^2$, the first term dominates the others
in the limit $|k|^2/\chi_{k-j}^{}\to\infty$, the latter being implied (since $|j|\le\kpg$) by the hypothesis $\lim_{|k|\to\infty}\chi_k/|k|^\alpha=0$ for every $\alpha\ge2$.
We thus have for large $|k|$
\begin{equation}
   \lim_{t\to\infty}\EV\,\vtt_k(t)\overline{\vtt_k(t)}
	\simeq |k|^{-4}\sum\nolimits_j\gamma_j^2(k\wedge j)^2|k-j|^{2\beta},
\end{equation}
which is precisely its static value \eqref{q:vfi2}.
Following the proof of Lemma~\ref{t:vfi0} from \eqref{q:vfi2} to \eqref{q:aux30}, we obtain \eqref{q:EVvtt}.
Repeated integration by parts of \eqref{q:aux37}, assuming sufficiently smooth $\Phi$, keeping in mind that odd derivatives of $\Phi$ all vanish, and putting the resulting expression in \eqref{q:Evttk} give \eqref{q:EVvttk}.

The computation for the variance \eqref{q:Varvtt} goes along similar lines.
We have
\begin{equation}
   \Var\bigl(\tssum_k\,|\vtt_k|^2\bigr)
	= \EV\,\bigl(\tssum_k\,|\vtt_k|^2\bigr)^2 - \bigl(\EV\,\tssum_k|\vtt_k|^2\bigr)^2
\end{equation}
with
\begin{equation}
   \EV\,\bigl(\tssum_k\,|\vtt_k|^2\bigr)^2
	= \tssum_k\EV\,|\vtt_k|^4 + \tssum_{k\ne l}\,\EV\,(|\vtt_k|^2|\vtt_l|^2).
\end{equation}
Each term in the first sum is of the form (with no need to average over $\xi$)
\begin{align*}
   &\EV\,(\vtt_k\bar\vtt_k\vtt_k\bar\vtt_k)
	= \sum_{ijmn} (k\wedge i)(k\wedge j)(k\wedge m)(k\wedge n)\gamma_i\gamma_j\gamma_m\gamma_n\\
	&\qquad |k-i|^\beta|k-j|^\beta|k-m|^\beta|k-n|^\beta\ex^{\im(\xi_i-\xi_j+\xi_m-\xi_n)}\\
	&\qquad \int_0^t\int_0^t\int_0^t\int_0^t \ex^{(s_i+s_j+s_m+s_n-4t)|k|^2} \EV\,\cis(\phi_{k-i}(s_i)-\phi_{k-j}(s_j)+\phi_{k-m}(s_m)-\phi_{k-n}(s_n)) \ds_i \ds_j \ds_m \ds_n
\end{align*}
Reasoning as we did in the proof of Lemma~\ref{t:vfi0}, the integral is non-zero only in the following two cases: $i=j$ and $m=n$, and $i=n$ and $j=m$.
Denoting the integral as $I_{var}$, it factorises into an identical pair, giving us
\begin{equation}
   I_{var} = \Bigl(\int_0^t\int_0^t \ex^{(s_i+s_j-2t)|k|^2}\Phi_{k-j}(s_i-s_j) \ds_i\ds_j\Bigr)^2.
\end{equation}
This is exactly (square of) the integral in \eqref{q:aux30}, giving us
\begin{equation}
   I_{var} = |k|^{-8}\bigl(1 + {\sf o}(\chi_{k-j}^{}/|k|^2)\bigr).
\end{equation}
As with the expected value, $\EV\,|\vtt_k|^4$ reduces exactly to the static case in the large $|k|$ limit.
The same holds with the computation for $\EV\,|\vtt_k|^2|\vtt_l|^2$ for $k\ne l$, which is very similar (having to take $\EV$ over $\xi$) and not presented here.
Summing over $k$ as in the static case gives us \eqref{q:Varvtt}.

Finally we turn to the remainder $\delta\tht$.
Since $|\Phi_k(s)|\le1$, time-dependent velocity can only ``weaken'' the tracer $\tht$ compared to the static case (where $\Phi_k(s)\equiv1$), not strengthen it.
The proof of Lemma~\ref{t:thtn} holds line-by-line for the present time-dependent case, giving us \eqref{q:bddthtt}.
\end{proof}


\section*{Appendix: Error Bound on Lattice Effects}

We derive an error bound for the second approximation in \eqref{q:aux30} and later in the proofs of Lemma~\ref{t:vfi0} where a sum over $k$ is replaced by an integral.
Writing $k=|k|(\cos\alpha_k,\sin\alpha_k)$ and $j=|j|(\cos\alpha_j,\sin\alpha_j)$, we have
\begin{equation}\label{q:aux50}
   \sum_{\kappa\le|k|<2\kappa} (k\wedge j)^2|k|^{2\beta}
	= \sum_{\kappa\le|k|<2\kappa} |j|^2|k|^{2\beta+2}\sin^2(\alpha_k-\alpha_j)
\end{equation}
where $j$ (and thus $\alpha_j$) are henceforth fixed.
Now let
\begin{equation}
   f(x,y)=(x^2+y^2)^{\beta+1}\sin^2(\tan^{-1}(y/x)-\alpha_j).
\end{equation}
Taylor's theorem gives
\begin{equation}
   f(k_x+x,k_y+y) = f(k_x,k_y) + x\,\dy_xf(k_x+\xi_x,k_y+\xi_y) + y\,\dy_yf(k_x+\xi_x,k_y+\xi_y)
\end{equation}
with $|\xi_x|\le|x|<1$ and $|\xi_y|\le|y|<1$.
We compute
\begin{equation*}
   \dy_xf(x,y) = (x^2+y^2)^\beta[2(\beta+1)x\sin^2(\tan^{-1}(y/x)-\alpha_j) - y\sin(2\tan^{-1}(y/x)-2\alpha_j)]
\end{equation*}
and an analogous expression for $\dy_yf$.
Bounding these gives us
\begin{equation}
   |\gb f(k_x+\xi_x,k_y+\xi_y)| \le (1 + |2\beta-1|)|k|^{2\beta+1}
\end{equation}
for all $|\xi_x|$, $|\xi_y|\le1$.
Integrating over a unit cell gives
\begin{equation}
   \int_{k_x}^{k_x+1}\!\!\int_{k_y}^{k_y+1} |f(k_x+x,k_y+y) - f(k_x,k_y)| \;\mathrm{d}x\;\mathrm{d}y \le (2 + |4\beta-2|)|k|^{2\beta+1}.
\end{equation}
Summing this over $\sim3\pi\kappa^2$ cells for which $\kappa\le|k|<2\kappa$ gives us the bound $6\pi(1+|2\beta-1|)\kappa^{2\beta+3}$.
To this one must add the error arising from the fact that the annulus and the cells do not coincide, which is also $\lesssim c_\beta\kappa^{2\beta+3}$.
We thus have
\begin{equation}
   \Bigl|\sum_{\kappa\le|k|<2\kappa} (k\wedge j)^2|k|^{2\beta}
	- \int_\kappa^{2\kappa}\!\!\int_0^{2\pi} (j_x\sin\varpi - j_y\cos\varpi)^2 r^{2\beta+3} \;\mathrm{d}\varpi\dr\Bigr|
\	\le c(\beta)|j|^2\kappa^{2\beta+1}.
\end{equation}


\bigskip\hbox to\hsize{\qquad\hrulefill\qquad}

\bigskip\hbox to\hsize{\qquad\hrulefill\qquad}\medskip


\medskip

\begin{thebibliography}{10}

\bibitem{batchelor:59}
{\sc G.~K. Batchelor}, {\em Small-scale variation of convected quantities like
  temperature in turbulent fluid, part~1: general discussion and the case of
  small conductivity}, J. Fluid Mech., 5 (1959), pp.~113--133.

\bibitem{batchelor-howells-townsend:59}
{\sc G.~K. Batchelor, I.~D. Howells, and A.~A. Townsend}, {\em Small-scale
  variation of convected quantities like temperature in turbulent fluid,
  part~2: the case of large conductivity}, J. Fluid Mech., 5 (1959),
  pp.~134--139.

\bibitem{chasnov:91}
{\sc J.~R. Chasnov}, {\em Simulation of the inertial--conductive subrange},
  Phys. Fluids, A3 (1991), pp.~1164--1168.

\bibitem{chasnov-canuto-rogallo:88}
{\sc J.~R. Chasnov, V.~M. Canuto, and R.~S. Rogallo}, {\em Turbulence spectrum
  of a passive temperature field: results of a numerical simulation}, Phys.
  Fluids, 31 (1988), pp.~2065--2067.

\bibitem{chen-kraichnan:98}
{\sc S.~Chen and R.~H. Kraichnan}, {\em Simulations of a randomly advected
  passive scalar field}, Phys. Fluids, 10 (1998), pp.~2867--2884.

\bibitem{corrsin:51}
{\sc S.~Corrsin}, {\em On the spectrum of isotropic temperature fluctuations in
  an isotropic turbulence}, J. Appl. Phys., 22 (1951), pp.~469--473.

\bibitem{gotoh-watanabe-suzuki:11}
{\sc T.~Gotoh, T.~Watanabe, and Y.~Suzuki}, {\em Universality and anisotropy in
  passive scalar fluctuations in turbulence with uniform mean gradient}, J.
  Turbulence, 12 (2011), p.~N48.

\bibitem{holzer-siggia:94}
{\sc M.~Holzer and E.~D. Siggia}, {\em Turbulent mixing of a passive scalar},
  Phys. Fluids, 6 (1994), pp.~1820--1837.

\bibitem{jolly-dw:ttur}
{\sc M.~S. Jolly and D.~Wirosoetisno}, {\em Energy spectra and passive tracer
  cascades in turbulent flows}, J. Math. Phys., 59 (2018), p.~073104.
\newblock arXiv:1611.08027.

\bibitem{kimura-kraichnan:93}
{\sc Y.~Kimura and R.~H. Kraichnan}, {\em Statistics of an advected passive
  scalar}, Phys. Fluids, A5 (1993), pp.~2264--2277.

\bibitem{kraichnan:68}
{\sc R.~H. Kraichnan}, {\em Small-scale structure of a scalar field convected
  by turbulence}, Phys. Fluids, 11 (1968), pp.~945--963.

\bibitem{ngan-pierrehumbert:00}
{\sc K.~Ngan and R.~T. Pierrehumbert}, {\em Spatially correlated and
  inhomogeneous random advection}, Phys. Fluids, 12 (2000), pp.~822--834.

\bibitem{obukhov:49}
{\sc A.~M. Obukhov}, {\em Structure of the temperature field in turbulent
  flows}, Izv. Akad. Nauk SSSR, ser. Geogr. Geofiz., 13 (1949), pp.~58--63.

\bibitem{ogorman-pullin:05}
{\sc P.~A. O'Gorman and D.~I. Pullin}, {\em Effect of {Schmidt} number on the
  velocity--scalar cospectrum in isotropic turbulence with a mean scalar
  gradient}, J. Fluid Mech., 532 (2005), pp.~111--140.

\bibitem{shraiman-siggia:00}
{\sc B.~I. Shraiman and E.~D. Siggia}, {\em Scalar turbulence}, Nature, 405
  (2000), pp.~639--646.

\bibitem{vallis:aofd}
{\sc G.~K. Vallis}, {\em Atmospheric and oceanic fluid dynamics}, Cambridge
  Univ. Press, 2006.

\bibitem{warhaft:00}
{\sc Z.~Warhaft}, {\em Passive scalars in turbulent flows}, Ann. Rev. Fluid
  Mech., 32 (2000), pp.~203--240.

\bibitem{yeung-sreenivasan:13}
{\sc P.~K. Yeung and K.~R. Sreenivasan}, {\em Spectrum of passive scalars of
  high molecular diffusivity in turbulent mixing}, J. Fluid Mech., 716 (2013),
  p.~R14.

\bibitem{yeung-sreenivasan:14}
\leavevmode\vrule height 2pt depth -1.6pt width 23pt, {\em Direct numerical
  simulation of turbulent mixing at very low {Schmidt} number with a uniform
  mean gradient}, Phys. Fluids, 26 (2014), p.~015107.

\end{thebibliography}
\end{document}